\documentclass{amsart}
\usepackage{amssymb}
\usepackage{hyperref}
\usepackage[capitalise]{cleveref}
\usepackage{thm-restate}
\usepackage{tikz}
\allowdisplaybreaks

\title{Finite free convolutions via Weingarten calculus}

\author{Jacob Campbell}
\address{Department of Pure Mathematics, University of Waterloo, Waterloo, Ontario N2L 3G1, Canada}
\email{j48campb@uwaterloo.ca}

\author {Zhi Yin}
\address{Institute of Advanced Study in Mathematics, Har\-bin Institute of Technology, Har\-bin 150006, China}
\email{hustyinzhi@163.com}

\newtheorem{thm}{Theorem}[section]
\newtheorem{prop}[thm]{Proposition}
\newtheorem{lem}[thm]{Lemma}

\theoremstyle{definition}
\newtheorem{defn}[thm]{Definition}
\newtheorem{nota}[thm]{Notation}

\newcommand{\bC}{\mathbb{C}}
\newcommand{\bE}{\mathbb{E}}
\newcommand{\bT}{\mathbb{T}}

\newcommand{\bfZ}{\mathbf{Z}}
\newcommand{\bfi}{\mathbf{i}}
\newcommand{\bfj}{\mathbf{j}}
\newcommand{\bfp}{\mathbf{p}}
\newcommand{\bfpp}{\mathbf{pp}}

\newcommand{\Ind}{\mathrm{Ind}}
\newcommand{\Sym}{\mathrm{Sym}}
\newcommand{\Wg}{\mathrm{Wg}}
\newcommand{\diag}{\mathrm{diag}}
\newcommand{\dil}{\mathrm{dil}}
\newcommand{\sgn}{\mathrm{sgn}}
\newcommand{\triv}{\mathrm{triv}}

\newcommand{\sfe}{\mathsf{e}}

\newcommand*\myrect[1]{%
    \begin{tikzpicture}
        \node[draw,rectangle,inner sep=0pt] {#1};
    \end{tikzpicture}}
\newcommand{\recsum}{\mathbin{\textnormal{\myrect{++}}}}

\begin{document}

\begin{abstract}
    We consider the three finite free convolutions for polynomials studied in a recent paper by Marcus, Spielman, and Srivastava. Each can be described either by direct explicit formulae or in terms of operations on randomly rotated matrices. We present an alternate approach to the equivalence between these descriptions, based on combinatorial Weingarten methods for integration over the unitary and orthogonal groups. A key aspect of our approach is to identify a certain \emph{quadrature property}, which is satisfied by some important series of subgroups of the unitary groups (including the groups of unitary, orthogonal, and signed permutation matrices), and which yields the desired convolution formulae.
\end{abstract}

\maketitle

\section{Introduction}

In \cite{MSS2019}, A. Marcus, D. A. Spielman, and N. Srivastava studied three convolution-type operations on polynomials of degree $d$:

\begin{defn}[{\cite[Definitions 1.1 \& 1.4 \& 1.7]{MSS2019}}]
    For polynomials $p(x),q(x) \in \bC[x]$ with degree at most $d$, say 
    \[ p(x) = \sum_{k=0}^d x^{d-k} (-1)^k a_k \text{ and } q(x) = \sum_{k=0}^d x^{d-k} (-1)^k b_k \text{,} \]
    we make the following definitions:
    \begin{enumerate}
        \item the \emph{symmetric additive convolution} of $p(x)$ and $q(x)$ is defined by
            \[ p(x) \boxplus_d q(x) := \sum_{k=0}^d x^{d-k} (-1)^k \left( \sum_{i+j=k} \frac{(d-i)! (d-j)!}{d! (d-k)!} a_i b_j \right) \text{;} \]
        \item the \emph{symmetric multiplicative convolution} of $p(x)$ and $q(x)$ is defined by
            \[ p(x) \boxtimes_d q(x) := \sum_{k=0}^d x^{d-k} (-1)^k \left( \frac{k! (d-k)!}{d!} a_k b_k \right) \text{;} \]
        \item the \emph{asymmetric additive convolution} of $p(x)$ and $q(x)$ is defined by
            \[ p(x) \recsum_d q(x) := \sum_{k=0}^d x^{d-k} (-1)^k \left( \sum_{i+j=k} \left( \frac{(d-i)! (d-j)!}{d! (d-k)!} \right)^2 a_i b_j \right) \text{.} \]
    \end{enumerate}
\end{defn}

These are also called \emph{finite free convolutions}, due to their connection with free convolution in free probability theory, which is expounded upon in \cite{M2018}. The finite free convolutions can be defined explicitly as above, in terms of formulae involving the coefficients, but the motivation is their forthcoming description in terms of expected characteristic polynomials of the corresponding operations with randomly rotated matrices. We use the following notation for characteristic polynomials:
\[ c_x(A) = \det(xI - A) = \sum_{k=0}^d x^{d-k} (-1)^k \sfe_k(A) \]
is the characteristic polynomial of the matrix $A \in M_d(\bC)$ in the variable $x$, and $\sfe_k(A)$ is the $k$-th elementary symmetric function of the eigenvalues of $A$.

\begin{thm}[{\cite[Theorems 1.2 \& 1.5 \& 1.8]{MSS2019}}]
    \label{thm:Convolution}
    Let $p(x),q(x) \in \bC[x]$ be polynomials with degree $d$.
    \begin{description}
        \item[symmetric additive] If $A,B \in M_d(\bC)$ are normal with $c_x(A) = p(x)$ and $c_x(B) = q(x)$, then
            \begin{equation}
                \label{eq:sqsum}
                p(x) \boxplus_d q(x) = \bE_U c_x(A+UBU^*)
            \end{equation}
            where $U$ is a random $d \times d$ unitary matrix.
        \item[symmetric multiplicative] If $A,B \in M_d(\bC)$ are normal with $c_x(A) = p(x)$ and $c_x(B) = q(x)$, then
            \begin{equation}
                \label{eq:sqmult}
                p(x) \boxtimes_d q(x) = \bE_U c_x(AUBU^*)
            \end{equation}
            where $U$ is a random $d \times d$ unitary matrix.
        \item[asymmetric additive] If $A,B \in M_d(\bC)$ are matrices with $c_x(AA^*) = p(x)$ and $c_x(BB^*) = q(x)$, then
            \begin{equation}
                \label{eq:recsum}
                p(x) \recsum_d q(x) = \bE_{U,V} c_x((A+UBV)(A+UBV)^*)
            \end{equation}
            where $U$ and $V$ are independent random $d \times d$ unitary matrices.
    \end{description}
\end{thm}

In this paper we present an alternate proof of \cref{thm:Convolution}, and we also prove the analogous results for random \emph{orthogonal} matrices instead of random \emph{unitary} matrices. Our approach revolves around the Weingarten calculus for integration on the unitary and orthogonal groups; in particular, we take advantage of the connection between Weingarten functions and the combinatorial representation theory of symmetric groups. Using the orthogonality relations for irreducible characters of $S_k$ in the unitary case and for the zonal spherical functions of the Gelfand pair $(S_{2k},H_k)$ in the orthogonal case, we obtain combinatorial identities for the relevant Weingarten functions which allow us to greatly reduce the apparent complexity of direct computations of the relevant expected characteristic polynomials.

A very interesting aspect of \cite{MSS2019} is what one might call a quadrature phenomenon: the random unitary or orthogonal matrices in the definitions of the finite free convolutions can be replaced with random signed permutation matrices, thus discretizing these operations. They isolate a property that they call minor-orthogonality, which makes this quadrature work and which is satisfied by the group $H_d$ of signed permutation matrices. Our techniques also yield quadrature results, with a \emph{quadrature property} which plays a similar role as minor-orthogonality:

\begin{restatable}{defn}{qp}
    A compact subgroup $G \leq U_d$ satisfies the \emph{quadrature property} if
    \[ \sum_{\sigma \in S_k} \sgn(\sigma) \int_G \prod_{i=1}^k u_{i \bfp(i)} \overline{u_{\sigma(i) \bfp(i)}} \, dU = \begin{cases}
        \frac{(d-k)!}{d!} & \text{if } \bfp \text{ is injective} \\
        0 & \text{otherwise}
    \end{cases} \]
    for $\bfp : [k] \to [d]$, for all $0 \leq k \leq d$.
\end{restatable}

In these terms, the main results are that various groups have the quadrature property:

\begin{thm}
    \label{thm:QuadratureGroups}
    The following groups have the quadrature property:
    \begin{enumerate}
        \item the group $U_d$ of $d \times d$ unitary matrices;
        \item the group $O_d$ of $d \times d$ orthogonal matrices;
        \item the groups $H_d^s = \widehat{\bfZ_s} \wr S_d$, for $2 \leq s \leq \infty$. (Here, $\bfZ_{\infty}$ is just $\bfZ$, and $S_d$ is the group of $d \times d$ permutation matrices.) In particular, the case $s=2$ recovers the group of $d \times d$ signed permutation matrices.
    \end{enumerate}
\end{thm}

We prove \cref{thm:Convolution} in terms of this quadrature property; specifically, we prove the following theorem for the expected elementary symmetric functions $\sfe_k$ of the eigenvalues of the randomly rotated matrices, which is easily seen to recover \cref{thm:Convolution} with $G = U_d$:

\begin{restatable}[{\cite[Theorems 2.10 \& 2.12 \& 2.14]{MSS2019}}]{thm}{convquad}
    \label{thm:ConvolutionQuadrature}
    Let $G \leq U_d$ be a compact subgroup with the quadrature property. Then
    \begin{enumerate}
        \item for $0 \leq k \leq d$, we have
            \[ \int_G \sfe_k(A+UBU^*) \, dU = \sum_{i+j=k} \frac{(d-i)!(d-j)!}{d!(d-k)!} \sfe_i(A) \sfe_j(B) \]
            for normal $A,B \in M_d(\bC)$;
        \item for $0 \leq k \leq d$, we have
            \[ \int_G \sfe_k(AUBU^*) \, dU = \frac{k! (d-k)!}{d!} \sfe_k(A) \sfe_k(B) \]
            for normal $A,B \in M_d(\bC)$;
        \item for $0 \leq k \leq d$, we have
            \begin{align*}
                &\quad \iint_G \sfe_k((A+UBV)(A+UBV)^*) \, dU \, dV \\
                &= \sum_{i+j=k} \left( \frac{(d-i)!(d-j)!}{d!(d-k)!} \right)^2 \sfe_i(AA^*) \sfe_j(BB^*)
            \end{align*}
            for $A,B \in M_d(\bC)$.
    \end{enumerate}
\end{restatable}

Since the first version of this paper was posted, there has been a substantial update to \cite{MSS2019}, and we have incorporated some of the changes. At some points we still refer to the first version \cite{MSS2015}, namely for some technical details about the asymmetric additive convolution which do not differ in our proofs.

The paper is organized as follows. In Section 2 we review some preliminaries on permutations and partitions, the combinatorial representation theory of $S_k$ and the Gelfand pair $(S_{2k},H_k)$, and the Weingarten calculus for integration on the unitary and orthogonal groups. In Section 3 we isolate, from the definitions of the finite free convolutions, the computations we need to handle, and we arrive at our quadrature property. Sections 4, 5, and 6 are devoted to the proofs that $U_d$, $O_d$, and $H_d^s$, respectively, have the quadrature property. In the final Section 7 we prove \cref{thm:ConvolutionQuadrature}, thus finishing the proof of \cref{thm:Convolution} in terms of the quadrature property.

\section{Preliminaries}

In this section we review some necessary preliminaries from combinatorics and random matrix theory. We mostly follow the notation of \cite{CM2009}.

\subsection{Elementary combinatorics}

For an integer partition $\lambda \vdash k$, we may either write $\lambda = (\lambda_1,\lambda_2,\ldots)$, or when convenient, $\lambda= (1^{m_1(\lambda)}, 2^{m_2(\lambda)}, \ldots)$ where $m_i(\lambda)$ is the multiplicity of $i$ in $\lambda$; write $\ell(\lambda) = \sum_{i \geq 1} m_i(\lambda)$ for the length of $\lambda$. We denote by $\mu_{\sigma} \vdash k$ the cycle type of a permutation $\sigma \in S_k$, and for $\rho \vdash k$, we write $z_{\rho} := \prod_{i \geq 1} i^{m_i(\rho)} m_i(\rho)!$. Let us collect some properties that we will use:

\begin{lem}
    \label{lem:PermutationsPartitions}
    Let $\rho \vdash k$.
    \begin{enumerate}
        \item The number of permutations $\sigma \in S_k$ with $\mu_{\sigma} = \rho$ is $\frac{k!}{z_{\rho}}$.
        \item We have $z_{2 \rho} = 2^{\ell(\rho)} z_{\rho}$, where $2 \rho = (2 \rho_1,2 \rho_2,\ldots)$.
        \item For $\sigma \in S_k$, we have $\sgn(\sigma) = (-1)^{k - \ell(\mu_{\sigma})}$.
    \end{enumerate}
\end{lem}
\begin{proof}
    For (1), see e.g. \cite[Proposition 1.3.2]{SBook}. For (2), we have
    \[ z_{2\rho} = \prod_{i \geq 1} (2i)^{m_i(\rho)} m_i(\rho)! = 2^{\ell(\rho)} \prod_{i \geq 1} i^{m_i(\rho)} m_i(\rho)! = 2^{\ell(\rho)} z_{\rho} \]
    since $\ell(\rho) = \sum_{i \geq 1} m_i(\rho)$. For (3), a cycle of length $i$ is a product of $i-1$ transpositions:
    \[ (j_1,\ldots,j_i) = (j_1,j_2) \cdots (j_{i-1},j_i) \text{.} \]
    So if $\mu_{\sigma} = (1^{m_1},2^{m_2},\ldots)$, then $\sigma$ can be written as a product of
    \[ \sum_{i \geq 1} (i-1) m_i = \sum_{i \geq 1} i m_i - \sum_{i \geq 1} m_i =  k - \ell(\mu_{\sigma}) \]
    transpositions, which gives $\sgn(\sigma) = (-1)^{k - \ell(\mu_{\sigma})}$.
\end{proof}

The number of permutations in $S_k$ with exactly $i$ cycles is called the \emph{unsigned Stirling number of the first kind} with parameters $k$ and $i$ and is denoted by $c(k,i)$. These numbers are related to the $k$-th \emph{rising factorial}, which is
\[ x^{(k)} = x (x+1) \cdots (x+k-1) \text{,} \]
by the following lemma:

\begin{lem}
    \label{lem:Stirling}
    We have
    \[ \sum_{i=0}^k  c(k, i) x^i = x^{(k)} \text{.} \]
\end{lem}
\begin{proof}
    See e.g. \cite[Proposition 1.3.7]{SBook}.
\end{proof}

Recall that a \emph{partition} of a set $X$ is a collection of disjoint subsets, called \emph{blocks} of the partition, whose union is $X$. Let $P(k)$ be the set of partitions of the set $[k] = \{ 1,\ldots,k \}$ and for $\pi \in P(k)$, write $| \pi |$ for the number of blocks of $\pi$. Write $P_2(k)$ for the set of partitions of $[k]$ whose blocks are all of size $2$; note that this is only non-empty when $k$ is even.

The ordering of $P(k)$ is defined by letting $\pi \leq \sigma$ if each block of $\pi$ is a subset of a block of $\sigma$. This ordering makes $P(k)$ a lattice. For $\bfi : [k] \to [d]$, write $\ker(\bfi)$ for the element of $P(k)$ whose blocks are the equivalence classes of the relation defined by $s \sim t$ if and only if $\bfi(s) = \bfi(t)$; in other words $\ker(\bfi)$ is the partition of $[k]$ whose blocks are the ``level sets" of the multi-index $\bfi$. With this notation, $\pi \leq \ker(\bfi)$ if and only if $\bfi(s) = \bfi(t)$ whenever $s$ and $t$ are in the same block of $\pi$; in other words the multi-index $\bfi$ labels the blocks of $\pi$ in a consistent way.

There is a natural embedding of $P_2(2k)$ into $S_{2k}$ which we will use extensively as in \cite{CM2009}: each $\pi \in P_2(2k)$ can be written uniquely in the form
\[ \{ \{ \pi(1),\pi(2) \},\ldots,\{ \pi(2k-1),\pi(2k) \} \} \]
with $\pi(2i-1) < \pi(2i)$ for $1 \leq i \leq k$ and $\pi(1) < \cdots < \pi(2k-1)$, and the embedding is
\begin{equation}
    \label{eq:Embed}
    \pi \mapsto \begin{pmatrix}
        1 & \cdots & 2k \\
        \pi(1) & \cdots & \pi(2k)
    \end{pmatrix} \in S_{2k} \text{.}
\end{equation}

\subsection{Combinatorial representation theory}

The representation theory of symmetric groups is well known to be described by the combinatorics of integer partitions and Young diagrams; a nice reference on this theory is \cite{CSTBook}. In particular, the irreducible representations of $S_r$ are canonically labeled, say as $V^{\lambda}$, by the integer partitions $\lambda \vdash r$. Write $\chi^{\lambda}$ for the character of $V^{\lambda}$, and $\chi_{\rho}^{\lambda}$ for the value of $\chi^{\lambda}$ on the conjugacy class in $S_r$ labeled by $\rho \vdash r$.

\subsubsection{The symmetric group and Schur functions}

An important general feature of the representation theory of finite groups is that the characters of irreducible representations are orthogonal, and for $S_k$ this takes the following form:

\begin{thm}[Orthogonality relations for $\chi^{\lambda}$]
    \label{thm:SymmetricOrthogonal}
    For $\lambda,\mu \vdash k$, we have
    \[ \frac{1}{| S_k |} \sum_{\sigma \in S_k} \chi^{\lambda}(\sigma) \chi^{\mu}(\sigma) = \sum_{\rho \vdash k} \frac{1}{z_{\rho}} \chi_{\rho}^{\lambda} \chi_{\rho}^{\mu} = \begin{cases}
        1 & \text{if } \lambda = \mu \\
        0 & \text{otherwise}
    \end{cases} \text{.} \]
\end{thm}
\begin{proof}
    See e.g. \cite[Corollary 1.3.7]{CSTBook}.
\end{proof}

The graded algebra whose $k$-th component is the space of class functions on $S_k$ is identified with the graded algebra of symmetric functions, see e.g. \cite[Section I.7]{MBook}, and the symmetric functions $s_{\lambda}$ corresponding to the irreducible characters $\chi^{\lambda}$ are called the \emph{Schur functions}. We just need a particular value of $s_{\lambda}$:

\begin{prop}[Hook-content formula]
    We have
    \[ s_{\lambda}(1^d) = \frac{\chi^{\lambda}(1)}{k!} \prod_{(i,j) \in \lambda} (d+j-i) \]
    for $\lambda \vdash k$ and $k \leq d$.
\end{prop}
\begin{proof}
    See e.g. \cite[Theorem 4.3.3]{CSTBook}.
\end{proof}

\subsubsection{The Gelfand pair $(S_{2k},H_k)$}

If $G$ is a finite group and $K$ is a subgroup of $G$, recall that $(G,K)$ is called a \emph{Gelfand pair} if the trivial representation of $K$ induces a multiplicity-free representation of $G$. Write $H_k$ for the centralizer of $(1,2) \cdots (2k-1,2k) \in S_{2k}$; this is called the \emph{hyperoctahedral group} and has order $2^k k!$. The hyperoctahedral group $H_k$ may be alternately described as the group of signed permutations of $k$ symbols, as the symmetry group of a $k$-dimensional hypercube, or as the wreath product $S_2 \wr S_k$. It is well known (see e.g. \cite[VII.2.2]{MBook}) that $(S_{2k},H_k)$ is a Gelfand pair.

Associated to a Gelfand pair is its family of \emph{zonal spherical functions}. These are defined by taking the characters of the irreducible representations contained in $\Ind_K^G(\triv)$ and averaging them over $K$, which we will make more precise below. By \cite[VII.2.4]{MBook} the irreducible representations of $S_{2k}$ contained in $\Ind_{H_k}^{S_{2k}}(\triv)$ are precisely the ones labeled by $2 \lambda$ for $\lambda \vdash k$, so we make the following concrete definition:

\begin{defn}[Zonal spherical functions]
    For $\lambda = (\lambda_1,\lambda_2,\ldots) \vdash k$, define $\omega^{\lambda} : S_{2k} \to \bC$ by
    \[ \omega^{\lambda}(\sigma) = \frac{1}{| H_k |} \sum_{\zeta \in H_k} \chi^{2\lambda} (\sigma \zeta) \]
    for $\sigma \in S_{2k}$. This is called the \emph{zonal spherical function} of the Gelfand pair $(S_{2k},H_k)$ corresponding to $\lambda$.
\end{defn}

Let us single out the values of a particular $\omega^{\lambda}$:

\begin{lem}
    \label{lem:ZonalValue}
    For $\rho \vdash k$, we have
    \[ \omega_\rho^{1^k} = \frac{(-1)^{k-\ell(\rho)}}{2^{k-\ell(\rho)}} \text{.} \]
\end{lem}
\begin{proof}
    See e.g. \cite[Example VII.2.2.b]{MBook}.
\end{proof}

In the analogy between $\chi^{\lambda}$ and $\omega^{\lambda}$, the relation of $\chi^{\lambda}$ with the conjugacy classes of $S_k$ corresponds to the relation of $\omega^{\lambda}$ with the double cosets $H_k \sigma H_k$ of $H_k$ in $S_{2k}$:

\begin{defn}
    For $\sigma \in S_{2k}$, define a graph $\Gamma(\sigma)$ as follows:
    \begin{itemize}
        \item the vertices are $1,\ldots,2k$;
        \item the edges connect $2i-1$ with $2i$ and $\sigma(2i-1)$ with $\sigma(2i)$ for $1 \leq i \leq k$.
    \end{itemize}
    The connected components of $\Gamma(\sigma)$ are cycles of even lengths, and dividing those lengths by $2$, we get an integer partition $\Xi(\sigma)$ of $k$ which is called the \emph{coset type} of $\sigma$.
\end{defn}

By e.g. \cite[VII.2.1.i]{MBook}, the coset type labels the double cosets of $H_k$ in $S_{2k}$. So for $\rho \vdash k$ we write $H_{\rho}$ for the corresponding double coset and then $S_{2k} = \bigsqcup_{\rho \vdash k} H_{\rho}$. By \cite[VII.2.3]{MBook} the cardinality of a double coset $H_{\rho}$ is
\[ | H_{\rho} | = \frac{| H_k |^2}{z_{2\rho}} = \frac{| H_k |^2}{2^{\ell(\rho)} z_{\rho}} \text{.} \]
Clearly the zonal spherical functions are constant on double cosets, so we write $\omega_{\rho}^{\lambda}$ for the value of $\omega^{\lambda}$ on $H_{\rho}$. For us, the critical property of the zonal spherical functions is that we still have the orthogonality relations analogous to \cref{thm:SymmetricOrthogonal}:

\begin{thm}[Orthogonality relations for $\omega^{\lambda}$]
    \label{thm:ZonalOrthogonal}
    For $\lambda,\mu \vdash k$, we have
    \[ \sum_{\rho \vdash k} \frac{1}{z_{2\rho}} \omega_\rho^\lambda \omega_\rho^\mu = \begin{cases}
        \frac{h(2\lambda)}{| H_k |^2} & \text{if } \lambda = \mu \\
        0 & \text{otherwise}
    \end{cases} \text{,} \]
    where $h(2\lambda)$ is the product of the hook lengths in $2 \lambda$.
\end{thm}
\begin{proof}
    See e.g. \cite[VII.2.15']{MBook}.
\end{proof}

We need an analogue of the relationship between the irreducible characters $\chi^{\lambda}$ and the Schur functions $s_{\lambda}$. To this end, one identifies the graded algebra whose $k$-th component is the space of functions on $S_{2k}$ which are constant on the double cosets of $H_k$, with the graded algebra of symmetric functions. Then, the \emph{zonal polynomial} $Z_{\lambda}$ is the symmetric function corresponding to $\omega^{\lambda}$ in this identification. Again, we only need a particular value of $Z_{\lambda}$; write $(x)_k := x (x-1) \cdots (x-k+1)$ for the $k$-th \emph{falling factorial}.

\begin{prop}
    We have
    \[ Z_{\lambda}(1^d) = \prod_{(i,j) \in \lambda} (d+2j-i-1) \]
    for $\lambda \vdash k$ and $k \leq d$. In particular, we have $Z_{1^k}(1^d) = (d)_k$.
\end{prop}
\begin{proof}
    See e.g. \cite[VII.2.25]{MBook}.
\end{proof}

\subsection{Random matrices and the Weingarten calculus}

We denote by $U_d$ and $O_d$ the compact groups of $d \times d$ unitary and orthogonal matrices, respectively; a \emph{random $d \times d$ unitary} or \emph{orthogonal matrix} is a random element of $U_d$ or $O_d$, respectively, sampled according to the groups' respective Haar probability measures. For $1 \leq i,j \leq d$, let $u_{ij} : M_d \to \bC$ be the $(i,j)$-th matrix coordinate function, i.e. the function which picks out the $(i,j)$-th entry of a matrix.

The \emph{Weingarten calculus} is a family of combinatorial techniques for integration over certain classical matrix groups, named after D. Weingarten due to his pioneering work \cite{W1978} which concerned the group $SO_d$. Ideas which first appeared there were systematized and developed for the unitary group by Collins in \cite{C2003}, and later for the orthogonal and symplectic groups by Collins-\'{S}niady in \cite{CS2006}. One may prefer to take the perspective of e.g. Banica-Speicher in \cite{BS2009}, in which the Weingarten calculus follows from the construction of combinatorial models of the representation categories of so-called \emph{easy} groups.

\subsubsection{Integration on the unitary group}

The main theorem on integration over $U_d$ is the following, which is due to Collins and \'{S}niady in \cite{C2003,CS2006}. The matrix $\Wg_{k,d}^U$, indexed by $S_{2k}$, is constructed from the invariant theory of $U_d$.

\begin{thm}[Unitary Weingarten calculus]
    \label{thm:UnitaryWeingarten}
    For $k,k' \geq 1$ and $\bfi,\bfj : [k] \to [d]$ and $\bfi',\bfj' : [k'] \to [d]$, the integral
    \[ \int_{U_d} u_{\bfi(1) \bfj(1)} \cdots u_{\bfi(k) \bfj(k)} \overline{u_{\bfi'(1) \bfj'(1)}} \cdots \overline{u_{\bfi'(k') \bfj'(k')}} \, dU \]
    is
    \[ \sum_{\substack{\pi,\sigma \in S_k \\ \bfi = \bfi' \circ \pi \\ \bfj = \bfj' \circ \sigma}} \Wg_{k,d}^U(\pi,\sigma) \]
    when $k = k'$, and it is $0$ otherwise.
\end{thm}

We single out an expression for $\Wg^U$ in terms of the characters of $S_{2k}$ which we use in a critical way:

\begin{thm}[{\cite[Proposition 2.3]{CS2006}}]
    \label{thm:SchurWeingarten}
    For $\pi,\sigma \in S_{2k}$, we have
    \[ \Wg_{k,d}^U(\pi,\sigma) = \frac{1}{(k!)^2} \sum_{\substack{\lambda \vdash k \\ \ell(\lambda) \leq d}} \frac{\chi^{\lambda}(1)^2}{s_{\lambda}(1^d)} \chi^{\lambda}(\pi^{-1} \sigma) \text{.} \]
    In particular, $\Wg_{k,d}^U(\pi,\sigma)$ only depends on the cycle type of $\pi^{-1} \sigma$.
\end{thm}

\subsubsection{Integration on the orthogonal group}
 
Now let us state the main theorem on integration over $O_d$, which is due to Collins-\'{S}niady in \cite[Corollary 3.4]{CS2006}. The matrix $\Wg_{k,d}^O$, indexed by $P_2(2k)$, is constructed from the invariant theory of $O_d$, although in a somewhat different way from the unitary case.

\begin{thm}[Orthogonal Weingarten calculus]
    \label{thm:OrthogonalWeingarten}
    For $k \geq 1$ and $\bfi,\bfj : [2k] \to [d]$, we have
    \[ \int_{O_d} u_{\bfi(1) \bfj(1)} \cdots u_{\bfi(2k) \bfj(2k)} \, dU = \sum_{\substack{\pi,\sigma \in P_2(2k) \\ \pi \leq \ker(\bfi) \\ \sigma \leq \ker(\bfj)}} \Wg_{k,d}^O(\pi,\sigma) \]
    where the integral is with respect to the Haar probability measure of $O_d$.
\end{thm}

Again, we require an expression for $\Wg^O$ in terms of the representation theory of $S_{2k}$; here, the analogue of \cref{thm:SchurWeingarten} is in terms of the zonal spherical functions of the Gelfand pair $(S_{2k},H_k)$. Recall that the value of $\chi^{2 \lambda}$ at $1$, which is the dimension of the irreducible representation of $S_{2k}$ labeled by $2 \lambda$, is
\[ \chi^{2 \lambda}(1) = \left| \{ \text{standard Young tableaux of shape } 2\lambda \} \right| = \frac{(2k)!}{h(2\lambda)} \]
where the last equality is by the hook-length formula \cite[Theorem 4.2.14]{CSTBook}.

\begin{thm}[{\cite[Theorem 3.1]{CM2009}}]
    \label{thm:GelfandWeingarten}
    For $\pi,\sigma \in P_2(2k),$ we have
    \[ \Wg_{k,d}^O(\pi,\sigma) = \frac{2^k k!}{(2k)!} \sum_{\substack{\lambda \vdash k \\ \ell(\lambda) \leq d}} \frac{\chi^{2 \lambda}(1)}{Z_{\lambda}(1^d)} \omega^{\lambda}(\pi^{-1} \sigma) \]
    where $P_2(2k)$ is embedded into $S_{2k}$ as in \eqref{eq:Embed}. In particular, $\Wg_{k,d}^O(\pi,\sigma)$ only depends on the coset type $\Xi(\pi^{-1} \sigma)$ of $\pi^{-1} \sigma$.
\end{thm}

\section{Quadrature property and motivation}
\label{sec:Quadrature}

In this section we will explain our \emph{quadrature property} and where it comes from. Recall

\qp*

This is best motivated by just diving in to the symmetric additive case and seeing what Haar integrals must be handled. We can assume without loss of generality that $A$ and $B$ are diagonal: since they are normal, they can be diagonalized, say as $A = V_A D_A V_A^*$ and $B = V_B D_B V_B^*$ for some unitary $V_A$ and $V_B$ and some diagonal $D_A$ and $D_B$. Then we have
\begin{align*}
    c_x(A+UBU^*) &= c_x(V_A D_A V_A^* + U V_B D_B V_B^* U^*) \\
    &= c_x(D_A + V_A^* U V_B D_B V_B^* U^* V_A) \\
    &= c_x(D_A + (V_A^* U V_B) D_B (V_A^* U V_B)^*)
\end{align*}
and by invariance of Haar measure, we have
\[ \bE_U c_x(A + U B U^*) = \bE_U c_x(D_A + U D_B U^*) \text{.} \]
So write $A = \diag(a_1,\ldots,a_d)$ and $B = \diag(b_1,\ldots,b_d)$, and $W := A + UBU^*$.

To approach this expected characteristic polynomial, observe that
\[ c_x(W) = \sum_{k=0}^d x^{d-k} (-1)^k \sfe_k(W) \]
and
\[ \sfe_k(W) = \sum_{\substack{S \subseteq [d] \\ | S | = k}} \det(W(S,S)) \]
where for $S,T \subseteq [d]$ we denote by $W(S,T)$ the submatrix of $W$ consisting of the rows indexed by $S$ and the columns indexed by $T$. So we want to look at $\bE_U \det(W(S,S))$ for $S \subseteq [d]$ with $| S | = k$. 

\begin{lem}
    We have
    \begin{align*}
        \bE_U \det(W(S,S)) &= \sum_{R \subseteq S} \left( \prod_{i \in R} a_i \right) \sum_{\bfp : S \setminus R \to [d]} \left( \prod_{i \in S \setminus R} b_{\bfp(i)} \right) \\
        &\qquad \sum_{\sigma \in \Sym(S \setminus R)} \sgn(\sigma) \bE_U \left( \prod_{i \in S \setminus R} u_{i \bfp(i)} \overline{u_{\sigma(i) \bfp(i)}} \right)
    \end{align*}
    for $S \subseteq [d]$ with $| S | = k$, for $0 \leq k \leq d$.
\end{lem}
\begin{proof}
    We have
    \begin{align*}
        &\quad \det(W(S, S)) \\
        &= \sum_{\sigma \in \Sym(S)} \sgn(\sigma) \prod_{i \in S} \left( a_i \delta_{i,\sigma(i)} + \sum_{p=1}^d u_{ip} b_p \overline{u_{\sigma(i) p}} \right) \\
        &= \sum_{\sigma \in \Sym(S)} \sgn(\sigma) \sum_{R \subseteq S} \left( \left( \prod_{i \in R} a_i \delta_{i,\sigma(i)} \right) \left( \prod_{i \in S \setminus R} \left( \sum_{p=1}^d u_{ip} b_p \overline{u_{\sigma(i) p}} \right) \right) \right) \\
        &= \sum_{R \subseteq S} \sum_{\sigma \in \Sym(S \setminus R)} \sgn(\sigma) \left( \prod_{i \in R} a_i \right) \left( \prod_{i \in S \setminus R} \left( \sum_{p=1}^d u_{ip} b_p \overline{u_{\sigma(i) p}} \right) \right) \text{.}
    \end{align*}
    Switching the product and sum, we have
    \begin{align*}
        &\quad \bE_U \left( \prod_{i \in S \setminus R} \left( \sum_{p=1}^d u_{i p} b_p \overline{u_{\sigma(i) p}} \right) \right) \\
        &= \sum_{\bfp : S \setminus R \to [d]} \bE_U \left( \prod_{i \in S \setminus R} u_{i \bfp(i)} \overline{u_{\sigma(i) \bfp(i)}} \right) \left( \prod_{i \in S \setminus R} b_{\bfp(i)} \right)
    \end{align*}
    and putting this back into the sum above, we get the desired formula.
\end{proof}

Thus we want to work with
\[ \sum_{\sigma \in S_k} \sgn(\sigma) \bE_U \left( \prod_{i=1}^k u_{i \bfp(i)} \overline{u_{\sigma(i) \bfp(i)}} \right) \]
for $\bfp : [k] \to [d]$, for $0 \leq k \leq d$, and the quadrature property does exactly this.

\section{The unitary case}

In this section we show that $U_d$ itself has the quadrature property, i.e. that
\begin{equation}
    \label{eq:UnitaryQuadrature}
    \sum_{\sigma \in S_k} \sgn(\sigma) \int_{U_d} \prod_{i=1}^k u_{i \bfp(i)} \overline{u_{\sigma(i) \bfp(i)}} \, dU = \begin{cases}
        \frac{(d-k)!}{d!} & \text{if } \bfp \text{ is injective} \\
        0 & \text{otherwise}
    \end{cases}
\end{equation}
for $\bfp : [k] \to [d]$, for all $0 \leq k \leq d$. To this end we use \cref{thm:SchurWeingarten} to reduce the computation to the following simple lemma:

\begin{lem}
    We have
    \[ \sum_{\sigma \in S_k} \sgn(\sigma) \chi^{\lambda}(\sigma) = \begin{cases}
        k! & \text{if } \lambda = 1^k \\
        0 & \text{otherwise}
    \end{cases} \]
    for $\lambda \vdash k$.
\end{lem}
\begin{proof}
    Observe that $\chi^{1^k}$ is just the sign character of $S_k$. So if $\lambda = 1^k$, we have
    \[ \sum_{\sigma \in S_k} \sgn(\sigma) \chi^{1^k}(\sigma) = \sum_{\sigma \in S_k} 1 = k! \]
    and otherwise, if $\lambda \neq 1^k$, then by \cref{thm:SymmetricOrthogonal} we have
    \[ \sum_{\sigma \in S_k} \sgn(\sigma) \chi^{\lambda}(\sigma) = \sum_{\sigma \in S_k} \chi^{1^k}(\sigma) \chi^{\lambda}(\sigma) = 0 \]
    so we are done.
\end{proof}

\begin{proof}[Proof of (1) in \cref{thm:QuadratureGroups}]
    By \cref{thm:UnitaryWeingarten} we have
    \begin{align*}
        \sum_{\sigma \in S_k} \sgn(\sigma) \int_{U_d} \prod_{i=1}^k u_{i \bfp(i)} \overline{u_{\sigma(i) \bfp(i)}} \, dU &= \sum_{\sigma \in S_k} \sgn(\sigma) \sum_{\substack{\pi,\tau \in S_k \\ 1 = \sigma \circ \pi \\ \bfp = \bfp \circ \tau}} \Wg_{k,d}^U(\pi,\tau) \\
        &= \sum_{\sigma \in S_k} \sgn(\sigma) \sum_{\substack{\tau \in S_k \\ \bfp = \bfp \circ \tau}} \Wg_{k,d}^U(\sigma^{-1},\tau) \text{.}
    \end{align*}
    If $\bfp$ is not injective, say there are some $i,j \in [k]$ with $i \neq j$ and $\bfp(i) = \bfp(j)$, we want to identify pairs of summands which cancel each other out, i.e. for each $\sigma \in S_k$ we want a corresponding $\sigma' \in S_k$ with $\sgn(\sigma') = -\sgn(\sigma)$ and
    \[ \sum_{\substack{\tau \in S_k \\ \bfp = \bfp \circ \tau}} \Wg_{k,d}^U(\sigma^{-1},\tau) = \sum_{\substack{\tau \in S_k \\ \bfp = \bfp \circ \tau}} \Wg_{k,d}^U((\sigma')^{-1},\tau) \text{.} \]
    To this end let $\sigma' = \sigma \cdot (i,j)$, so that $\sgn(\sigma') = -\sgn(\sigma)$. Moreover, since $\Wg_{k,d}^U(\sigma^{-1},\tau)$ only depends on the cycle type of $\sigma \tau$, we have
    \begin{align*}
        \sum_{\substack{\tau \in S_k \\ \bfp = \bfp \circ \tau}} \Wg_{k,d}^U((\sigma')^{-1},\tau) &= \sum_{\substack{\tau \in S_k \\ \bfp = \bfp \circ \tau}} \Wg_{k,d}^U((i,j) \sigma^{-1},\tau) \\
        &= \sum_{\substack{\tau \in S_k \\ \bfp = \bfp \circ \tau}} \Wg_{k,d}^U(\sigma^{-1},(i,j) \tau) \\
        &= \sum_{\substack{\tau \in S_k \\ \bfp = \bfp \circ \tau}} \Wg_{k,d}^U(\sigma^{-1},\tau)
    \end{align*}
    as the condition $\bfp = \bfp \circ \tau$ is invariant under translation of $\tau$ by $(i,j)$. Thus we have shown that when $\bfp$ is not injective, the summands in \cref{eq:UnitaryQuadrature} cancel each other out and the sum is $0$.

    If $\bfp$ is injective, then the only $\tau \in S_k$ with $\bfp = \bfp \circ \tau$ is $\tau = 1$ so by \cref{thm:SchurWeingarten} we have
    \begin{align*}
        &\quad \sum_{\sigma \in S_k} \sgn(\sigma) \sum_{\substack{\tau \in S_k \\ \bfp = \bfp \circ \tau}} \Wg_{k,d}^U(\sigma^{-1},\tau) \\
        &= \sum_{\sigma \in S_k} \sgn(\sigma) \Wg_{k,d}^U(\sigma^{-1},1) \\
        &= \sum_{\sigma \in S_k} \sgn(\sigma) \frac{1}{(k!)^2} \sum_{\substack{\lambda \vdash k \\ \ell(\lambda) \leq d}} \frac{\chi^{\lambda}(1)^2}{s_{\lambda}(1^d)} \chi^{\lambda}(\sigma) \\
        &= \frac{1}{(k!)^2} \sum_{\substack{\lambda \vdash k \\ \ell(\lambda) \leq d}} \left( \frac{k! \chi^{\lambda}(1)^2}{\chi^{\lambda}(1) \prod_{(i,j) \in \lambda} (d+j-i)} \right) \left( \sum_{\sigma \in S_k} \sgn(\sigma) \chi^{\lambda}(\sigma) \right) \\
        &= \left( \frac{\chi^{1^k}(1)}{k! \prod_{1 \leq i \leq k} (d+1-i)} \right) k! \\
        &= \frac{1}{\prod_{1 \leq i \leq k} (d-i+1)} = \frac{(d-k)!}{d!}
    \end{align*}
    and we are done.
\end{proof}

\section{The orthogonal case}

In this section we show that $O_d$ has the quadrature property, i.e. that
\begin{equation}
    \label{eq:OrthogonalQuadrature}
    \sum_{\sigma \in S_k} \sgn(\sigma) \int_{O_d} \prod_{i=1}^k u_{i \bfp(i)} u_{\sigma(i) \bfp(i)} \, dU = \begin{cases}
        \frac{(d-k)!}{d!} & \text{if } \bfp \text{ is injective} \\
        0 & \text{otherwise}
    \end{cases}
\end{equation}
for $\bfp : [k] \to [d]$, for all $0 \leq k \leq d$. To this end we use \cref{thm:GelfandWeingarten} to reduce the computation to the following lemma:

\begin{lem}
    \label{lem:ZonalQuadrature}
    We have
    \[ \sum_{\sigma \in S_k} \sgn(\sigma) \omega^{\lambda}(\sigma) = \begin{cases}
        \frac{(k+1)!}{2^k} & \text{if } \lambda = 1^k \\
        0 & \text{otherwise}
    \end{cases} \text{.} \]
    for $\lambda \vdash k$.
\end{lem}
\begin{proof}
    If $\lambda = 1^k$, then by (3) in \cref{lem:PermutationsPartitions}, \cref{lem:Stirling}, and \cref{lem:ZonalValue}, we have
    \begin{align*}
        \sum_{\sigma \in S_k} \sgn(\sigma) \omega^{1^k}(\sigma) &= \sum_{\sigma \in S_k} (-1)^{k-\ell(\mu_{\sigma})} \frac{(-1)^{k - \ell(\mu_{\sigma})}}{2^{k - \ell(\mu_{\sigma})}} \\
        &= \frac{1}{2^k} \sum_{\sigma \in S_k} 2^{\ell(\mu_{\sigma})} \\
        &= \frac{1}{2^k} \sum_{i=1}^k 2^i c(k,i) \\
        &= \frac{2^{(k)}}{2^k} \\
        &= \frac{(k+1)!}{2^k} \text{.}
    \end{align*}
    On the other hand, if $\lambda \neq 1^k$, then by \cref{lem:PermutationsPartitions}, \cref{lem:ZonalValue}, and  \cref{thm:ZonalOrthogonal}, for any $\lambda \neq 1^k$ we have
    \begin{align*}
        0 &= \sum_{\rho \vdash k} \frac{1}{z_{2\rho}} \omega_{\rho}^{\lambda} \omega_{\rho}^{1^k} \\
        &= \sum_{\rho \vdash k} \frac{1}{z_{2\rho}} \omega_\rho^\lambda \frac{(-1)^{k -\ell(\rho)}}{2^{k -\ell(\rho)}} \\
        &= \frac{1}{2^k} \sum_{\rho \vdash k} (-1)^{k-\ell(\rho)} 2^{\ell(\rho)} \frac{1}{z_{2\rho}} \omega_{\rho}^{\lambda} \\
        &= \frac{1}{2^k k!} \sum_{\rho \vdash k} (-1)^{k-\ell(\rho)} \frac{k!}{z_{\rho}} \omega_{\rho}^{\lambda} \\
        &= \frac{1}{2^k k!} \sum_{\sigma \in S_k} \sgn(\sigma) \omega^{\lambda}(\sigma)
    \end{align*}
    thus $\sum_{\sigma \in S_k} \sgn(\sigma) \omega_{\mu_{\sigma}}^{\lambda} = 0$.
\end{proof}

\begin{proof}[Proof of (2) in \cref{thm:QuadratureGroups}]
    By \cref{thm:OrthogonalWeingarten}, with $\bfi_{\sigma} := (1,\sigma(1),\ldots,k,\sigma(k))$ and $\bfpp := (\bfp(1),\bfp(1),\ldots,\bfp(k),\bfp(k))$, we have
    \begin{align*}
        \sum_{\sigma \in S_k} \int_{O_d} \prod_{i=1}^k u_{i \bfp(i)} u_{\sigma(i) \bfp(i)} \, dU &= \sum_{\sigma \in S_k} \sgn(\sigma) \sum_{\substack{\pi,\tau \in P_2(2k) \\ \pi \leq \ker(\bfi_{\sigma}) \\ \tau \leq \ker(\bfpp)}} \Wg_{k,d}^O(\pi,\tau) \\
        &= \sum_{\sigma \in S_k} \sgn(\sigma) \sum_{\substack{\tau \in P_2(2k) \\ \tau \leq \ker(\bfpp)}} \Wg_{k,d}^O(\ker(\bfi_{\sigma}),\tau)
    \end{align*}
    since the condition $\pi \leq \ker(\bfi_{\sigma})$ forces equality. If $\bfp$ is not injective, say there are some $i \neq j$ with $\bfp(i) = \bfp(j)$, we want to identify pairs of summands which cancel each other out, i.e. for each $\sigma \in S_k$ we want a corresponding $\sigma' \in S_k$ with $\sgn(\sigma') = -\sgn(\sigma)$ and
    \[ \sum_{\substack{\tau \in P_2(2k) \\ \tau \leq \ker(\bfpp)}} \Wg_{k,d}^O(\ker(\bfi_{\sigma}),\tau) = \sum_{\substack{\tau \in P_2(2k) \\ \tau \leq \ker(\bfpp)}} \Wg_{k,d}^O(\ker(\bfi_{\sigma'}),\tau) \text{.} \]
    To this end let $\sigma' = (i,j) \sigma$, which obviously satisfies $\sgn(\sigma') = -\sgn(\sigma)$. Moreover, we have $\ker(\bfj_{\sigma'}) = (i,j) \ker(\bfj_{\sigma})$ in the embedding \eqref{eq:Embed}, so with $\tau' = (i,j) \tau$, $\tau^{-1} \ker(\bfj_{\sigma'})$ and $(\tau')^{-1} \ker(\bfj_{\sigma})$ have the same coset type. Since the condition $\tau \leq \ker(\bfpp)$ is invariant under translation of $\tau$ by $(i,j)$, by \cref{thm:GelfandWeingarten} we have
    \begin{align*}
        \sum_{\substack{\tau \in P_2{2k} \\ \tau \leq \ker(\bfpp)}} \Wg_{k,d}^O(\ker(\bfi_{\sigma'}),\tau) &= \sum_{\substack{\tau \in P_2(2k) \\ \tau \leq \ker(\bfpp)}} \Wg_{k,d}^O(\ker(\bfi_{\sigma}),\tau') \\
        &= \sum_{\substack{\tau \in P_2(2k) \\ \tau \leq \ker(\bfpp)}} \Wg_{k,d}^O(\ker(\bfi_{\sigma}),\tau) \text{.}
    \end{align*}
    Thus we have shown that when $\bfp$ is not injective, the summands in \cref{eq:OrthogonalQuadrature} cancel each other out and the sum is $0$.

    If $\bfp$ is injective, then the condition $\tau \leq \ker(\bfpp)$ forces equality, so since
    \[ \chi^{2(1^k)}(1) = \frac{(2k)!}{k! (k+1)!} \]
    and $Z_{1^k}(1^d) = (d)_k$, by \cref{thm:GelfandWeingarten} and \cref{lem:ZonalQuadrature}, we have
    \begin{align*}
        &\quad \sum_{\sigma \in S_k} \sgn(\sigma) \sum_{\substack{\tau \in P_2(2k) \\ \tau \leq \ker(\bfpp)}} \Wg_{k,d}^O(\ker(\bfi_{\sigma}),\tau) \\
        &= \sum_{\sigma \in S_k} \sgn(\sigma) \Wg_{k,d}^O(\ker(\bfi_{\sigma}),\ker(\bfpp)) \\
        &= \sum_{\sigma \in S_k} \sgn(\sigma) \frac{2^k k!}{(2k)!} \sum_{\lambda \vdash k} \frac{\chi^{2 \lambda}(1)}{Z_{\lambda}(1^d)} \omega_{\mu_{\sigma}}^{\lambda} \\
        &= \sum_{\sigma \in S_k} \sgn(\sigma) \frac{2^k k!}{(2k)!} \frac{\chi^{2(1^k)}(1)}{Z_{1^k}(1^d)} \omega_{\mu_{\sigma}}^{1^k} + \sum_{\sigma \in S_k} \sgn(\sigma) \frac{2^k k!}{(2k)!} \sum_{\substack{\lambda \vdash k \\ \lambda \neq 1^k}} \frac{\chi^{2 \lambda}(1)}{Z_{\lambda}(1^d)} \omega_{\mu_{\sigma}}^{\lambda} \\
        &= \frac{2^k}{(k+1)! (d)_k} \sum_{\sigma \in S_k} \sgn(\sigma) \omega_{\mu_{\sigma}}^{1^k} + \sum_{\substack{\lambda \vdash k \\ \lambda \neq 1^k}} \frac{2^k k! \chi^{2 \lambda}(1)}{(2k)! Z_{\lambda}(1^d)} \sum_{\sigma \in S_k} \sgn(\sigma) \omega_{\mu_{\sigma}}^{\lambda} \\
        &= \frac{1}{(d)_k} = \frac{(d-k)!}{d!}
    \end{align*}
    and we are done.
\end{proof}

\section{The hyperoctahedral case}

In this brief section we show the elementary proof that the so-called \emph{hyperoctahedral series} of groups $H_d^s = \widehat{\bfZ_s} \wr S_d$ for $2 \leq s \leq \infty$, which consist of $d \times d$ ``signed" permutation matrices where the ``signs" are $s$-th roots of unity (or in the case $s=\infty$, the entire circle), have the quadrature property. This is similar to e.g. \cite[Lemma 2.6]{MSS2019} which is the case $s=2$, and the quadrature results for these groups in general are already known from \cite{HPS2018}. What we want to show is that
\[ \sum_{\sigma \in S_k} \sgn(\sigma) \int_{H_d^s} \prod_{i=1}^k u_{i \bfp(i)} \overline{u_{\sigma(i) \bfp(i)}} \, dU = \begin{cases}
    \frac{(d-k)!}{d!} & \text{if } \bfp \text{ is injective} \\
    0 & \text{otherwise}
\end{cases} \]
for $\bfp : [k] \to [d]$, for $0 \leq k \leq d$.

\begin{proof}[Proof of (3) in \cref{thm:QuadratureGroups}]
    For $s < \infty$, we have
    \begin{align*}
        &\quad \sum_{\sigma \in S_k} \sgn(\sigma) \int_{H_d^s} \prod_{i=1}^k u_{i \bfp(i)} \overline{u_{\sigma(i) \bfp(i)}} \, dU \\
        &= \sum_{\sigma \in S_k} \sgn(\sigma) \frac{1}{s^d d!} \sum_{\epsilon_1,\ldots,\epsilon_d \in \widehat{\bfZ_s}} \sum_{\tau \in S_d} \prod_{i=1}^k (\epsilon_i \delta_{\bfp(i)=\tau(i)}) (\overline{\epsilon_{\sigma(i)}} \delta_{\bfp(i)=\tau(\sigma(i))})
    \end{align*}
    and the non-zero summands are the ones with $\bfp(i) = \tau(i)$ and $\bfp(i) = \tau(\sigma(i))$ for $1 \leq i \leq k$. If $\bfp$ is not injective, then there is no $\tau \in S_d$ with $\bfp(i) = \tau(i)$ for $1 \leq i \leq k$, so the sum is $0$. On the other hand, if $\bfp$ is injective, there are $(d-k)!$ permutations $\tau \in S_d$ with $\bfp(i) = \tau(i)$ for $1 \leq i \leq k$, i.e. $\tau \in S_{d-k}$; similarly the condition $\bfp(i) = \tau(\sigma(i))$ forces $\sigma(i) = i$ for $1 \leq i \leq k$, i.e. $\sigma = 1$. So the sum above becomes
    \begin{align*}
        \frac{1}{s^d d!} \sum_{\epsilon_1,\ldots,\epsilon_d \in \widehat{\bfZ_s}} \sum_{\tau \in S_{d-k}} \prod_{i=1}^k | \epsilon_i |^2 &= \frac{(d-k)!}{s^d d!} \sum_{\epsilon_1,\ldots,\epsilon_d \in \widehat{\bfZ_s}} \prod_{i=1}^k 1 \\
        &= \frac{(d-k)!}{d!}
    \end{align*}
    since the last sum gives $s^d$ summands, which are copies of $1$. The case $s=\infty$ is similar, except with an integral over the $d$-torus $\bT^d$ instead of a sum over $d$ copies of $\widehat{\bfZ_s}$.
\end{proof}

\section{Convolution formulae and quadrature}

Finally we prove the convolution formulae; recall the generalization

\convquad*

\subsection{Symmetric additive convolution}

We have already done a large portion of the proof for the symmetric additive convolution in \cref{sec:Quadrature}.

\begin{proof}[Proof of (1) in \cref{thm:ConvolutionQuadrature}]
    Let us pick back up from the computations in \cref{sec:Quadrature}. Notice that since we assume $A = \diag(a_1,\ldots,a_d)$ and $B = \diag(b_1,\ldots,b_d)$, we have
    \[ \sfe_i(A) = \frac{1}{i!} \sum_{\substack{\bfp : [i] \to [d] \\ \text{injective}}} a_{\bfp(1)} \cdots a_{\bfp(i)} = \sum_{\substack{\bfp : [i] \to [d] \\ \bfp(1) < \cdots < \bfp(i)}} a_{\bfp(1)} \cdots a_{\bfp(i)} \]
    and similarly for $\sfe_j(B)$. So we have
    \begin{align*}
        &\quad \bE_U \det(W(S,S)) \\
        &= \sum_{R \subseteq S} \left( \prod_{i \in R} a_i \right) \sum_{\bfp : S \setminus R \to [d]} \left( \prod_{i \in S \setminus R} b_{\bfp(i)} \right) \\
        &\qquad \sum_{\sigma \in \Sym(S \setminus R)} \sgn(\sigma) \bE_U \left( \prod_{i \in S \setminus R} u_{i \bfp(i)} \overline{u_{\sigma(i) \bfp(i)}} \right) \\
        &= \sum_{R \subseteq S} \frac{(d-| S \setminus R |)!}{d!} \left( \prod_{i \in R} a_i \right) \sum_{\substack{\bfp : S \setminus R \to [d] \\ \text{injective}}} \left( \prod_{i \in S \setminus R} b_{\bfp(i)} \right) \\
        &= \sum_{R \subseteq S} \frac{| S \setminus R |! (d - | S \setminus R |)!}{d!} \det(A(R,R)) \sfe_{| S \setminus R |}(B)
    \end{align*}
    and then
    \begin{align*}
        \bE_U(\sfe_k(W)) &= \sum_{| S | = k} \bE_U(\det(W(S,S))) \\
        &= \sum_{|S|=k} \sum_{R \subseteq S} \frac{| S \setminus R |! (d - | S \setminus R |)!}{d!} \det(A(R, R)) \sfe_{| S \setminus R |}(B) \\
        &= \sum_{i+j=k} \frac{(d-i)! (d-j)!}{d! (d-k)!} \sfe_j(B) \sum_{|R|=i} \det(A(R, R)) \\
        &= \sum_{i+j=k} \frac{(d-i)! (d-j)!}{d! (d-k)!} \sfe_i(A) \sfe_j(B)
    \end{align*}
    so we are done.
\end{proof}

\subsection{Symmetric multiplicative convolution}

The computations for the symmetric multiplicative convolution are somewhat simpler:

\begin{proof}[Proof of (2) in \cref{thm:ConvolutionQuadrature}]
    As for (1), assume without loss of generality that $A$ and $B$ are diagonal with $A = \diag(a_1,\ldots,a_d)$ and $B = \diag(b_1,\ldots,b_d)$. Write $W = A U B U^*$ and $U = (u_{ij})_{i,j}$, so the $(i,j)$-th entry of $W$ is $a_i \sum_{p=1}^d u_{i p} b_p \overline{u_{j p}}$ and for a subset $S \subseteq [d]$ with $|S|=k$, we have
    \begin{align*}
        \det(W(S,S)) &= \sum_{\sigma \in \Sym(S)} \sgn(\sigma) \prod_{i \in S} \left( a_i \sum_{p=1}^d u_{ip} b_p \overline{u_{\sigma(i) p}} \right) \\
        &= \det(A(S,S)) \sum_{\sigma \in \Sym(S)} \sgn(\sigma) \prod_{i \in S} \left( \sum_{p=1}^d u_{ip} b_p \overline{u_{\sigma(i) p}} \right) \text{.}
    \end{align*}
    Switching the product and sum, we have
    \begin{align*}
        \bE_U(\det(W(S,S))) &= \det(A(S,S)) \sum_{\bfp : S \to [d]} \left( \prod_{i \in S} b_{\bfp(i)} \right) \\
        &\qquad \sum_{\sigma \in \Sym(S)} \sgn(\sigma) \bE_U \left( \prod_{i \in S} u_{i \bfp(i)} \overline{u_{\sigma(i) \bfp(i)}} \right) \\
        &= \det(A(S,S)) \sum_{\substack{\bfp : S \to [d] \\ \text{injective}}} \left( \prod_{i \in S} b_{\bfp(i)} \right) \frac{(d-| S |)!}{d!} \\
        &= \det(A(S,S)) \frac{k! (d-k)!}{d!} \sfe_k(B) \text{,}
    \end{align*}
    thus
    \[ \bE_U \sfe_k(AUBU^*) = \frac{k! (d-k)!}{d!} \sum_{|S|=k} \det(A(S,S)) \sfe_k(B) = \frac{k! (d-k)!}{d!} \sfe_k(A) \sfe_k(B) \]
    and we are done.
\end{proof}

\subsection{Asymmetric additive convolution}

The computations here, for the asymmetric additive convolution, are more involved than for the symmetric convolutions; we refer to \cite[Section 2.3.2]{MSS2015} at some points for details which do not concern our techniques. Let us first make some simplifying notation:

\begin{nota}
    For $A,B \in M_d(\bC)$, write
    \[ h[A,B](x) := \bE_{U,V} c_x((A+UBV)(A+UBV)^*) \]
    and
    \[ \dil(A) := \begin{pmatrix}
        0 & A \\
        A^* & 0
    \end{pmatrix} \text{.} \]
\end{nota}

\begin{proof}[Proof of (3) in \cref{thm:ConvolutionQuadrature}]
    Assume without loss of generality -- passing to singular value decompositions if necessary -- that $A$ and $B$ are diagonal with $A = \diag(a_1,\ldots,a_d)$ and $B = \diag(b_1,\ldots,b_d)$. We have
    \[ h[A,B](x^2) = \bE_{U,V} c_x(\dil(AU+VB)) \]
    so with $M := AU+VB$, we have
    \[ \bE_{U,V} c_x(\dil(M)) = \sum_{k=0}^{2d} x^{2d-k} \sum_{\substack{W \subseteq [2d] \\ | W | = k}} \sum_{\sigma \in \Sym(W)} \sgn(\sigma) \bE_{U,V} \left( \prod_{i \in W} \dil(M)_{i \sigma(i)} \right) \text{.} \]
    If $k$ is odd, then for any $W \subseteq [2d]$ with $|W| = k$ and for any $\sigma \in \Sym(W)$, there is some $i_0 \in W$ such that $\dil(M)_{i_0 \sigma(i_0)} =0$. So we may assume $k$ is even, say $k = 2l$. The coefficient of $x^{2d-2l}$ is
    \[ \sum_{\substack{S \text{ set of } l \text{ rows} \\ T \text{ set of } l \text{ columns}}} \sum_{\substack{\rho : S \to T \\ \text{bijection}}} \sum_{\sigma \in \Sym(S)} \sgn(\sigma) \bE_{U,V} \left( \prod_{i \in S} M_{i \rho(i)} \overline{M_{\sigma(i) \rho(i)}} \right) \]
    and we have
    \begin{align*}
        &\quad \bE_{U,V} \left( \prod_{i \in S} M_{i \rho(i)} \overline{M_{\sigma(i) \rho(i)}} \right) \\
        &= \bE_{U,V} \left( \prod_{i \in S} (a_i u_{i \rho(i)} + v_{i \rho(i)} b_{\rho(i)}) \overline{(a_{\sigma(i)} u_{\sigma(i) \rho(i)} + v_{\sigma(i) \rho(i)} b_{\rho(i)})} \right) \\
        &= \sum_{R \subset S} \bE_U \left( \prod_{i \in R}  a_i \overline{a_{\sigma(i)}} u_{i \rho(i)} \overline{u_{\sigma(i) \rho(i)}} \right) \bE_V \left( \prod_{i \in S \setminus R} b_{\rho(i)} \overline{b_{\rho(i)}} v_{i \rho(i)} \overline{v_{\sigma(i) \rho(i)}} \right) \\
        &= \sum_{R \subset S} \left( \prod_{i \in R} a_i \overline{a_{\sigma(i)}} \right) \left( \prod_{i \in S \setminus R} b_{\rho(i)} \overline{b_{\rho(i)}} \right) \int_G \prod_{i \in R} u_{i \rho(i)} \overline{u_{\sigma(i) \rho(i)}} \, dU \\
        &\hspace{40ex} \int_G \prod_{i \in S \setminus R} v_{i \rho(i)} \overline{v_{\sigma(i) \rho(i)}} \, dV
    \end{align*}
    since the cross terms vanish in the second-last line.

    Putting this back into the larger sum, we get
    \begin{align*}
        &\quad \sum_{\substack{S \text{ set of } l \text{ rows} \\ T \text{ set of } l \text{ columns}}} \sum_{\substack{\rho : S \to T \\ \text{bijection}}} \sum_{\sigma \in \Sym(S)} \sgn(\sigma) \bE_{U,V} \left( \prod_{i \in S} M_{i \rho(i)} \overline{M_{\sigma(i) \rho(i)}} \right) \\
        &= \sum_{\substack{S \text{ set of } l \text{ rows} \\ T \text{ set of } l \text{ columns}}} \sum_{\substack{\rho : S \to T \\ \text{bijection}}} \sum_{\sigma \in \Sym(S)} \sgn(\sigma) \sum_{R \subseteq S} \left( \prod_{i \in R} a_i \overline{a_{\sigma(i)}} \right) \left( \prod_{i \in S \setminus R} b_{\rho(i)} \overline{b_{\rho(i)}} \right) \\
        &\hspace{10ex} \int_G \prod_{i \in R} u_{i \rho(i)} \overline{u_{\sigma(i) \rho(i)}} \, dU \int_G \prod_{i \in S \setminus R} v_{i \rho(i)} \overline{v_{\sigma(i) \rho(i)}} \, dV \\
        &= \sum_{\substack{S \text{ set of } l \text{ rows} \\ T \text{ set of } l \text{ columns}}} \sum_{\substack{R \subseteq S \\ Z \subseteq T \\ |R| + |Z| = l}} \sum_{\substack{\rho : S \to T \\ \text{bijection} \\ \rho(S \setminus R) = Z}} \left( \prod_{i \in R} | a_i |^2 \right) \left( \prod_{i \in S \setminus R} | b_{\rho(i)} |^2 \right) \\
        &\qquad \sum_{\substack{\sigma_1 \in \Sym(R) \\ \sigma_2 \in \Sym(S \setminus R)}} \sgn(\sigma_1) \sgn(\sigma_2) \int_G \prod_{i \in R} u_{i \rho(i)} \overline{u_{\sigma_1(i) \rho(i)}} \, dU \\
        &\hspace{25ex} \int_G \prod_{i \in S \setminus R} v_{i \rho(i)} \overline{v_{\sigma_2(i) \rho(i)}} \, dV
    \end{align*}
    since only the only non-zero summands are the ones with $\sigma \in \Sym(R)$ or $\sigma \in \Sym(S \setminus R)$. Now by the quadrature property, this is equal to
    \begin{align*}
        &\quad \sum_{\substack{S \text{ set of } l \text{ rows} \\ T \text{ set of } l \text{ columns}}} \sum_{\substack{R \subseteq S \\ Z \subseteq T \\ |R| + |Z| = l}} \sum_{\substack{\rho : S \to T \\ \text{bijection} \\ \rho(S \setminus R) = Z}} \left( \prod_{i \in R} | a_i |^2 \right) \left( \prod_{i \in S \setminus R} | b_{\rho(i)} |^2 \right) \\
        &\hspace{35ex} \frac{(d-|R|)!}{d!} \frac{(d-|S \setminus R|)!}{d!} \\
        &= \sum_{r+z=l} \left( \frac{(d-r)! (d-z)!}{d! (d-l)!} \right)^2 \sfe_r(A A^*) \sfe_z(B B^*)
    \end{align*}
    where the last equality follows as in \cite[Section 2.3.2]{MSS2015}.
\end{proof}

\section*{Acknowledgements}

The authors wish to thank Alexandru Nica and Daniel Perales for introducing them to finite free convolutions, and for several enlightening discussions around these topics. J.C. would also like to thank Gavin Orok for his feedback on various aspects of this project. Z.Y wish to thank Benoit Collins and Ping Zhong for helpful discussions, and he also grateful to the Pure Mathematics Department of Waterloo University, where provided a fruitful environment to work on the project. This project was partially supported by NSFC No. 11771106.

\end{document}